\documentclass{amsart} 
\usepackage{amsthm,amssymb,amsmath,amsfonts,latexsym} 
\usepackage[pdftex]{graphicx,color} 
\usepackage{enumerate,subfigure} 

\newtheorem{theorem}{Theorem}[section] 
\newtheorem{lemma}[theorem]{Lemma}

\newtheorem{remark}[theorem]{Remark} 
\newtheorem{definition}[theorem]{Definition}

\newcommand{\R} {{\mathbb R}} 
\newcommand{\Z} {{\mathbb Z}}
\newcommand{\pr} {{\mathbb P}}

\newcommand{\innt} {\operatorname{int}}

\newcommand{\dx}{D_{x}} 
\newcommand{\dy}{D_{y}}

\newcommand{\tax}{T_{x}}

\newcommand{\cf}{\mathcal{C}} 

\newcommand{\snp}[1]{S^{+}_{#1}} 
\newcommand{\snm}[1]{S^{-}_{#1}}

\newcommand{\fx}{F_{x}}    
\newcommand{\tl}{T_{\Lambda}}
\newcommand{\ml}{\mu_{\Lambda}}

\newcommand{\w}[1] {\tilde{#1}}

\newcommand{\wt} {\w{\tau}}

\newcommand{\ta} {T}

\newcommand{\fp} {f_{+}} 
\newcommand{\fm} {f_{-}} 
\newcommand{\fo} {f_{1}} 
\newcommand{\fin} {f_{0}}

\newcommand{\bt} {\bar{\theta}} 
 
\newcommand{\hth}{\hat{\theta}} 
 
\newcommand{\corr} {\cos^{-1} \left(\frac{\cos \theta}{r}\right)} 
\newcommand{\sta} {(s,\theta)}

\newcommand{\D} { 
\partial}

\begin{document}

\title{Track billiards}

\author[L. Bunimovich]{Leonid A. Bunimovich} 
\address{ABC Math Program and School of Mathematics \\ 
Georgia Institute of Technology \\ Atlanta, GA 30332, U.S.A.} 
\email{bunimovh@math.gatech.edu} 
\author[G. Del Magno]{Gianluigi Del Magno} 
\address{Max Planck Institute for the Physics of Complex Systems \\
01187 Dresden, Germany} 
\email{delmagno@mpipks-dresden.mpg.de}

\date{\today}  
\begin{abstract}
	We study a class of planar billiards having the remarkable property that their phase space consists up to a set of zero measure of two invariant sets formed by orbits moving in opposite directions. The tables of these billiards are tubular neighborhoods of differentiable Jordan curves that are unions of finitely many segments and arcs of circles. We prove that under proper conditions on the segments and the arcs, the billiards considered have non-zero Lyapunov exponents almost everywhere. These results are then extended to a similar class of of 3-dimensional billiards. Finally, we find that for some subclasses of track billiards, the mechanism generating hyperbolicity is not the defocusing one that requires every infinitesimal beam of parallel rays to defocus after every reflection off of the focusing boundary.
\end{abstract}
                                
\maketitle     

\section{Introduction}
There are rather few examples of hyperbolic systems with several ergodic components, which are exactly described (for example, see \cite{w2,b3}). We study here a class of billiards whose phase space is up to a set of zero measure an union of two invariant sets consisting of orbits moving in opposite directions. The table of one of these billiards is a tubular neighborhood of differentiable Jordan curve that is a finite union of straight segments and arcs of circles. Since such a region looks somewhat like a track field, the billiards considered in this paper will be called track billiards. A simple example of a track billiard is obtained by cutting out a smaller stadium from a stadium (Fig. \ref{fig:tracks}(a)). 

In this paper, we prove that all the Lyapunov exponents of a track billiard are non-zero almost everywhere (hyperbolicity) provided that the segments and the arcs are sufficiently large, or that the segments and the width of the transverse section of the track are sufficiently large. We also generalize these results to 3-dimensional track billiards. There is no doubt that the hyperbolicity implies that the dynamics on each of the invariant sets formed by orbits moving in opposite directions is ergodic. This however will be the content of a forthcoming paper. 

It is worth pointing out that for some track billiards, the mechanism of hyperbolicity is not the defocusing one. This mechanism requires that after every reflection from the focusing part of the billiard boundary, a narrow beam of parallel rays must pass through a conjugate point, and become divergent before the next collision with the curved part of the boundary. Moreover, along a typical orbit, the average time of divergence along an orbit must exceed the average time of convergence. We found a class of track billiards that are hyperbolic, but do not satisfy the defocusing property. Namely, it is not true that every beam of parallel rays defocuses after reflecting off a focusing component and before the next reflection off a curved component of the boundary. To control such beams, we use the fact that the dynamics inside the curved part of a track is integrable.
 

Track billiards are also related to billiards in tubular regions, which model certain electronic devices in nanotechnology. Although, there are several works devoted to the study of the quantum properties of these billiards \cite{es,gj,cdfk,vpr}, we have found only a few works in the literature, which can give some insight on their classical properties \cite{hp,p}. Our results, may help fill in this gap.                                


The paper is organized as follows. In Section \ref{se:tracks}, we review some basic facts concerning billiard systems, introduce tracks billiards, and state the main result of this paper. The last part of Section \ref{se:tracks} contains some preliminary lemmas that are crucial in the proof of the hyperbolicity. In Section \ref{se:hyperbolicity}, we give the notions of focusing time and invariant cone field. Then, using a sort of generalized mirror formula for billiard trajectories crossing annular regions, we construct an eventually strictly invariant cone field for track billiards, whose existence implies hyperbolicity. Section \ref{se:hyperbolicity} contains also a discussion concerning the construction of the invariant cone field for a circular guide. Finally, in Section \ref{se:td}, the results obtained for 2-dimensional track billiards are extended to 3-dimensional track billiards.

\section{Track billiards}  
\label{se:tracks}

Let $ Q $ be a bounded domain of $ \R^{2} $ with piecewise differentiable boundary. The \emph{billiard in $ Q $} is the dynamical system arising from the motion of a point-particle inside $ Q $ obeying the following rules: the particle moves along straight lines at unit speed until it hits the boundary of $ Q $, at that moment, the particle gets reflected so that the angle of reflection equals the angle of incidence. 

\subsection{Definitions}   
\label{su:tracks}

The domain $ Q $ considered in this paper is a tubular neighborhood of a planar differentiable Jordan curve $ \gamma $ that is a finite union of segments and arcs of circles. Equivalently, we can say that $ Q $ is an union of finitely many building blocks of two types: circular guides and straight guides. A \emph{circular guide} is the region of an annulus with circles of radii $ r_{1}>r_{2}>0 $ contained inside a sector with central angle $ 0<\alpha<2\pi $ (see Fig. \ref{fig:guides}(a)). A \emph{straight guide} is simply a rectangle (see Fig. \ref{fig:guides}(b)). 
\begin{figure}[tb] 
	\centering \mbox{ \subfigure[Circular guide]{ 
	\includegraphics[width=.3
	\textwidth]{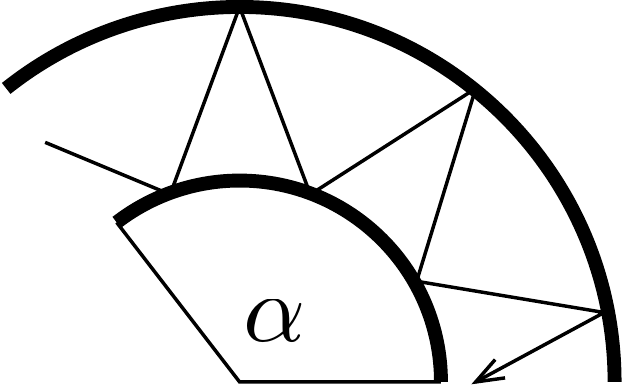}} \qquad \subfigure[Straight guide]{ 
	\includegraphics[width=.35
	\textwidth]{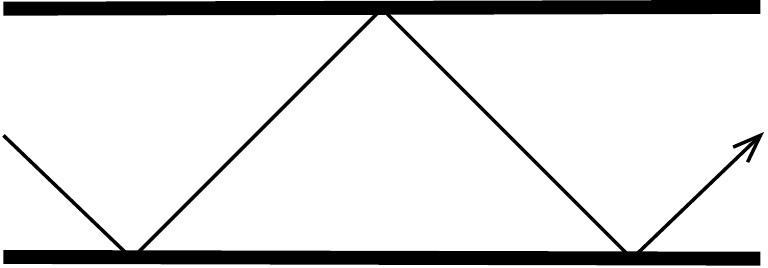}} } \caption{The two types of guides considered in this paper} \label{fig:guides} 
\end{figure}
The circular and straight guides must all have the same transverse width in order to fit together and form a domain $ Q $. Furthermore, we will always assume that any two circular guides of $ Q $ do not intersect (i.e., they are separated by at least one straight guide). Since $ Q $ resembles a track field, it is called a \emph{track}. Two examples of tracks are depicted in Fig. \ref{fig:tracks}. A billiard in $ Q $ is called a \emph{track billiard}. 
\begin{figure}[tb] 
	\centering \mbox{ \subfigure[]{ 
	\includegraphics[width=.4 
	\textwidth,height=3cm]{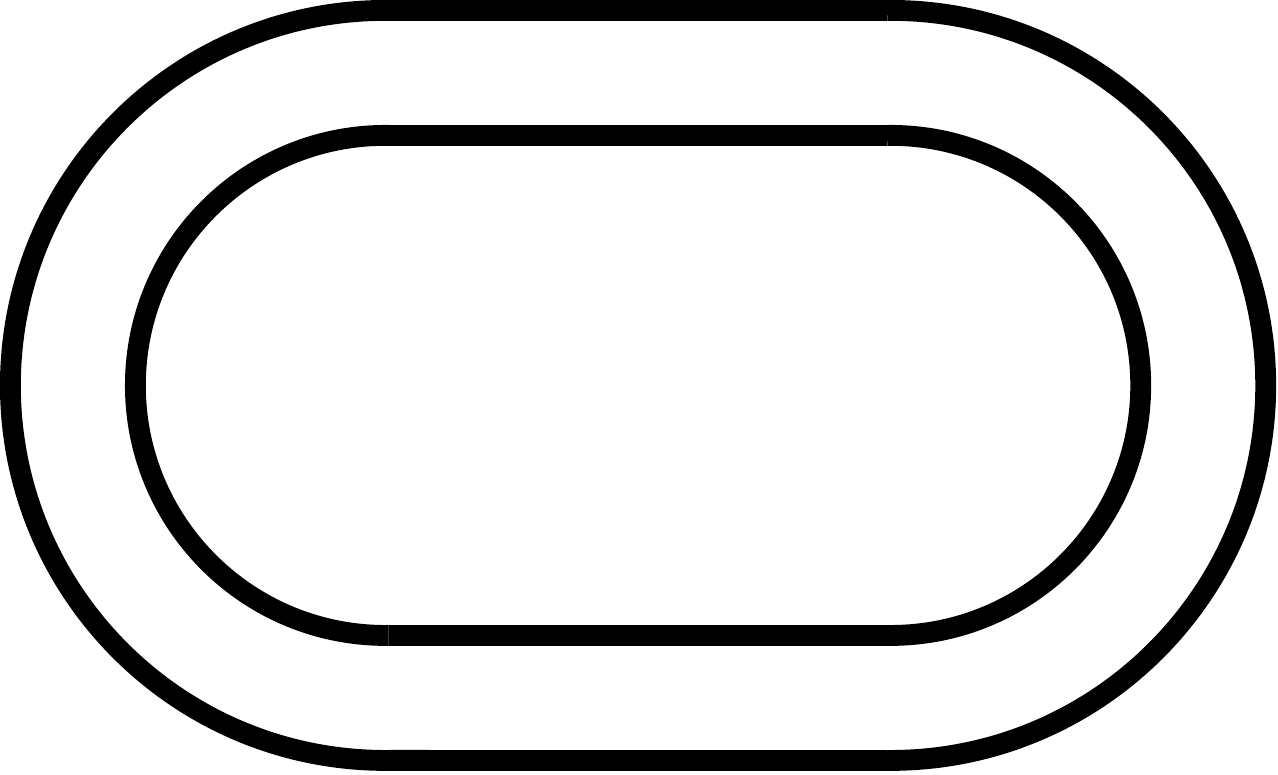}} \quad \subfigure[]{ 
	\includegraphics[width=.4
	\textwidth,height=3cm]{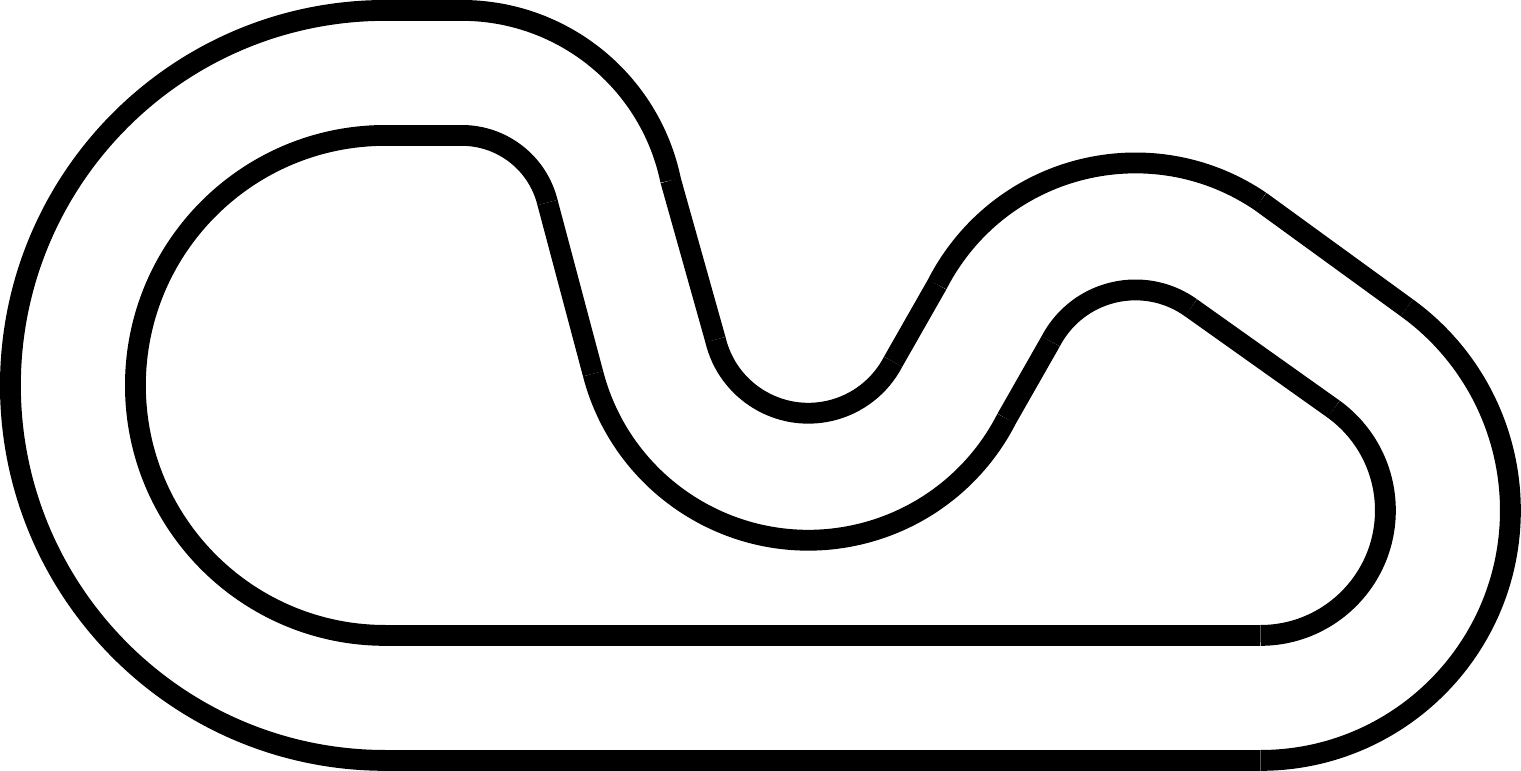}}} \caption{Two examples of tracks} \label{fig:tracks} 
\end{figure}  

For our purposes, the dynamics of a track billiard can be conveniently described by a discrete transformation called \emph{billiard map}, which is defined as follows. Let $ M $ be the set of all vectors $ (q,v) \in \ta_{1} \R^{2} $ such that $ q \in \D Q $ and $ \langle v,n(q)\rangle \geq 0 $, where $ n(q) $ is the normal vector to $ \D Q $ at $ q $ pointing inside $ Q $. Here $ \langle \cdot,\cdot \rangle $ is the standard dot product of $ \R^{2} $. The set $ M $ is easily seen to be a smooth manifold with boundary. Let $ \pi:M \to Q $ be the canonical projection defined by $ \pi(q,v)=q $ for $ (q,v) \in M $. If we view $ q $ and $ v $ as the position and the velocity of the particle after a collision with $ \D Q $, then $ M $ represents the collection of all possible post-collision states (collisions, for short) of the particle with $ \D Q $.  

Fix an orientation of the boundary $ \D Q $. A set of local coordinates for $ M $ is given by $ M \ni x \mapsto (s(x),\theta(x)) $, where $ s $ is the arclength parameter along the oriented boundary $ \D Q $, and $ 0 \le \theta \le \pi $ is the angle that the velocity of the particle forms with the oriented tangent of $ \D Q $. To specify an element $ x \in M $, we will use indifferently the notations $ x=(q,v) $ or $ x=\sta $. We endow $ M $ with the Riemannian metric $ ds^{2}+d\theta^{2} $ and the probability measure $ d \mu = (2|\D Q|)^{-1} \sin \theta ds d\theta $, where $ |\D Q| $ is the length of $ \D Q $. 

Denote by $ \D M $ the set of all vectors $ (q,v) \in M $ such that $ \langle v,n(q)\rangle = 0 $ or $ q $ is the endpoint of a straight segment of $ \D Q $. 
Let $ \innt M = M \setminus \D M $. For technical reasons, we define $ T $ only on collisions belonging to the smooth manifold (without boundary) $ \innt M $. The billiard map $ T:\innt M \to M $ is the transformation given by $ (q,v) \mapsto (q_{1},v_{1}) $, where $ (q,v) $ and $ (q_{1},v_{1}) $ are two consecutive collisions of the particle. Let us denote by $ \snp{1} $ the union of $ \D M $ and the subset of $ \innt M $ where $ T $ is not differentiable. It is easy to see that $ \snp{1} = \D M \cup T^{-1} \D M $. From the general results of \cite{ks}, it follows that $ \snp{1} $ is a compact set consisting of finitely many smooth compact curves that can intersect each other only at their endpoints. If we define $ \snm{1} = M \setminus T(M \setminus \snp{1}) $, then $ T $ is a diffeomorphism from $ M \setminus \snp{1} $ to its image $ M \setminus \snm{1} $, and preserves the measure $ \mu $ (see e.g. \cite{cfs,ks}).
 
The billiard dynamics is time-reversible. Indeed, the involution $ J:M \to M $ defined by $ J(s,\theta) = (s,\pi-\theta) $ for every $ (s,\theta) \in M $ has the property that $ J \circ T = T^{-1} \circ J $ everywhere on $ M \setminus \snp{1} $. Most of the time, we will use the notation $ -A $ instead of $ J A $, where $ A $ is a subset of $ M $. 

For every $ n>1 $, let us define $ \snp{n} = \snp{1} \cup T^{-1} \snp{1} \cup \cdots \cup T^{-n+1} \snp{1} $ and $ \snm{n} = \snm{1} \cup T \snm{1} \cup \cdots \cup T^{n-1} \snm{1} $. By the time-reversibility of the billiard dynamics, we have $ \snm{n} = -\snp{n} $ for every $ n>0 $. Let $ \snp{\infty} = \cup_{n>0} \snp{n} $ and $ \snm{\infty} = \cup_{n>0} \snm{n} $. Then $ \w{M} = M \setminus (\snm{\infty} \cup \snp{\infty}) $ is the set where all iterates of $ T $ are defined. Clearly, $ \mu(\snp{\infty})=\mu(\snm{\infty})=0 $ and $ \mu(\w{M})=1 $.

\subsection{Invariant sets and unidirectionality} 
\label{su:inv}         

We define $ v_{*} $ to be the component of the velocity the particle along the oriented tangent of $ \gamma $ at the point where the transverse section of $ Q $ passing through the position of the particle intersects $ \gamma $. This definition makes sense because $ Q = \{q+\delta n(q): q \in \gamma \text{ and } |\delta|<\epsilon\} $, where $ n(q) $ is the unit normal vector of $ \gamma $ at $ q $, and $ 2\epsilon $ is the transverse width of the guide. Since $ \gamma $ is differentiable, it easy to see that $ v_{*} $ is a continuous function of the time, and is constant inside a circular guide and a straight guide. Therefore $ sgn(v_{*}) $ is constant. Let $ L,R,N $ be the sets consisting of collisions $ \sta \in M $ such that $ \pi/2 < \theta < \pi $ and $ 0 < \theta < \pi/2 $ and $ \theta = \pi/2 $, respectively. Clearly, $ M = L \cup R \cup N $, and from the previous considerations, it follows that $ L,R,N $ are invariant\footnote{To be precise, we should write that $ L \cap \w{M},R \cap \w{M},N \cap \w{M} $ are invariant, because only on $ \w{M} $, every iterate of $ T $ is defined.}. Sometimes, this property is called \emph{unidirectionality}.

In a forthcoming paper, we will prove that for track billiards satisfying a condition slightly stronger than the one called H in this paper (see \eqref{eq:hypot} in Section \ref{se:hyperbolicity}, for the exact formulation), the sets $ L $ and $ R $ are the only invariant sets with measure between 0 and 1. In other words, the first return map of $ T $ to each set $ L $ and $ R $ is ergodic.

\begin{remark}
	In fact, the previous statement remains valid for a billiard in a tubular neighborhood of a differentiable Jordan curve in arbitrary dimension. To see this, first note that the set $ N = \{(q,v) \in M: v_{*} = 0 \} $ is invariant. Next, suppose that $ v_{*} > 0 $ initially, and that $ v_{*} < 0 $ at some later time, i.e., the particle changes the direction of its motion. Since $ v_{*} $ is a continuous function of the time, we see that $ v_{*} = 0 $ at some moment of time, which by the previous observation implies that $ v_{*} $ is identically equal to zero, giving a contradiction.
\end{remark}
  
\subsection{Main result} 
The map $ T $ is called \emph{(nonuniformly) hyperbolic} if all its Lyapunov exponents are non-zero almost everywhere on $ \w{M} $.

\begin{definition}
	We say that a circular guide is of \emph{type A} if $ \alpha \ge \pi $ (and no conditions on $ r_{1} $ and $ r_{2} $ are imposed), and is of \emph{type B} if $ r_{2}/r_{1}<1/2 $ (and no conditions on $ \alpha $ are imposed).
\end{definition}

We now state the main result of this paper.

\begin{theorem} 
	\label{th:main}
	Consider a track $ Q $ such that each of its circular guides is either of type A or B. The billiard map $ T $ in $ Q $ is hyperbolic provided that the straight guides of $ Q $ are sufficiently long.
\end{theorem} 
              
The outer circle and the inner circle of a circular guide are focusing and dispersing curves, respectively. To our knowledge, all the recipes for designing hyperbolic billiard domains including focusing and dispersing in their boundaries require these curves to be placed sufficiently apart \cite{b1,cm,dm,m2,w2,w3}. Since for guides of type A, there is no restriction on the distance between the outer and inner circles, Theorem \ref{th:main} tells us that there do exist hyperbolic billiard domains that violate the condition on the separation between focusing and dispersing boundary components. This easily implies that the mechanism generating hyperbolicity in these billiards is not the defocusing one, which requires that after a reflection off of a focusing curve, an infinitesimal family of parallel trajectories must focus and defocus before the next collision with the boundary of the billiard table.  
However, we have to point out that circular guides are very special domains being the billiards inside them integrable. Also, note that we still need to put circular guides sufficiently far away from each other in order to obtain hyperbolicity. While writing this paper, we learned that Bussolari and Lenci also constructed hyperbolic billiards (different than track billiards) that violate the aforementioned separation condition \cite{bl}. 

\subsection{Billiard dynamics in a circular guide} 
\label{su:dyncurve}
To prove Theorem \ref{th:main}, it is essential to investigate the billiard dynamics inside a circular guide. 

Consider a circular guide with outer and inner radii $ r_{1}=1 $ and $ 0<r_{2}=r<1 $, respectively. Note that by a proper rescaling, every circular guide can be transformed into such a guide. 
Denote by $ M_{1} $ the set of all collisions $ (q,v) $ such that $ q $ belongs to the outer circle of the guide. We will focus our attention on the transformation $ T_{1} $ that maps a collision with the outer circle to the next collision with the same circle (between these collisions, there may be a collision with the inner circle). If $ \sta $ belongs to the domain of $ T_{1} $, then 
\begin{equation}  
	\label{eq:map}
	T_{1} \sta = (s+2\delta(\theta),\theta),
\end{equation}
where $ 2\delta(\theta) $ is the central angle of the sector bounded by the two consecutive collisions with the outer circle (see Fig. \ref{fig:mapout}). 

Note that between $ \sta $ and $ T_{1} \sta $, there is a collision with the inner circle of the guide if and only if $ \theta \in [\bt,\pi-\bt] $, where $ \bt=\cos^{-1} r \in (0,\pi/2) $. For $ \theta \in [0,\bt) \cup (\pi-\bt,\pi] $, it is trivial to check that $ \delta(\theta)=\theta $. For $ \theta \in [\bt,\pi-\bt] $ instead, we immediately deduce from Fig. \ref{fig:mapout} that $ \delta(\theta) = \theta - \phi(\theta) $, where $ \phi(\theta) $ is the angle of the collision with the inner circle.   
\begin{figure}[tb] 
	\begin{center}
		\includegraphics[width=5cm]{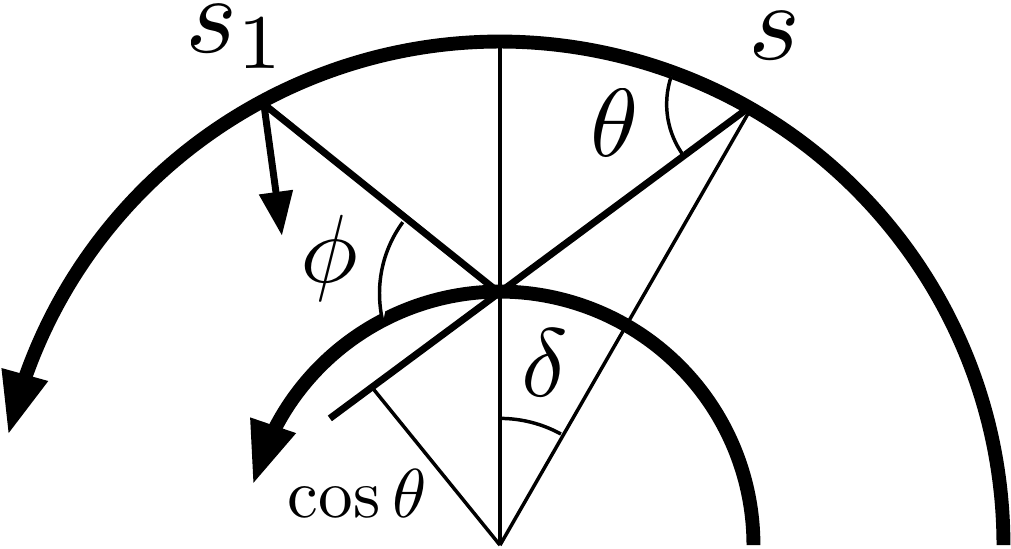}
	\end{center}
	\caption{Consecutive collisions inside a circular guide} 
	\label{fig:mapout}
\end{figure}
The relation between $ \theta $ and $ \phi $ is provided by the conservation of the angular momentum of the particle measured from the center of the circular guide, which reads as $ \cos \theta = r \cos \phi $. Putting all together, we obtain
\begin{equation}
	\label{eq:del} \delta(\theta)= 
	\begin{cases}
		\theta - \corr & \text{if } \theta \in [\bt,\pi-\bt], \\
		\theta & \text{if } \theta \in [0,\bt) \cup (\pi-\bt,\pi]. 
	\end{cases}
\end{equation}
The function $ \delta $ is differentiable on $ [0,\pi] \setminus \{\bt,\pi-\bt\} $, and $ \delta'(\theta) \to -\infty $, as $ \theta \to \bt^{+} $ or $ \theta \to (\pi-\bt)^{-} $. By abuse of notation, we define $ \delta(x) = \delta(\theta(x)) $ and $ \delta'(x) = \delta'(\theta(x)) $. 

\begin{definition} 
	For every $ x \in M_{1} $ such that $ \theta(x) \notin \{0,\bt,\pi-\bt,\pi\} $, denote by $ n_{1}(x) \ge 0 $ the times that the particle with initial state $ x $ hits the outer circle before leaving the guide.
\end{definition} 

Clearly, $ n_{1}(x) $ is finite. From \eqref{eq:map}, it then follows that
\begin{equation}
	\label{eq:der} D_{x} T^{n_{1}(x)}_{1} = 
	\begin{pmatrix}
		1 & 2 n_{1}(x) \delta'(x) \\
		0 & 1 
	\end{pmatrix}. 
\end{equation}

\subsection{Preliminary lemmas} 
\label{su:key} 
We now prove some facts that will play a crucial role in the proof the hyperbolicity of track billiards. The goal here is to estimate the quantity $ 2n_{1}(x)\delta'(x) $.

\begin{definition}
	\label{de:ent}   
	Let $ E_{1} $ be the set of all $ x \in M_{1} $ such that 
	\begin{enumerate}
		\item $ \theta(x) \notin \{0,\bt,\pi-\bt,\pi\} $, 
		\item $ x $ is an entering collision (i.e., $ n_{1}(-x)=0 $), 
		\item the particle with initial state $ x $ hits the outer circle before leaving the guide. 
	\end{enumerate}
\end{definition}

For every $ x \in E_{1} $, let us define
\begin{equation} 
	\label{eq:ome}
	\omega(x)=\alpha-2 n_{1}(x) \delta(x),
\end{equation} 
and
\begin{equation}
	\label{eq:chi} \chi(x) = 2 n_{1}(x) \delta'(x).
	\end{equation}

The next lemma is a trivial consequence of the fact that $ \delta'(x)=1 $ for all $ x \in E_{1} $ such that $ \theta(x) \in (0,\bt) \cup (\pi-\bt,\pi) $.
\begin{lemma}
	\label{le:disk}
	 If $ x \in E_{1} $ and $ \theta(x) \in (0,\bt) \cup (\pi-\bt,\pi) $, then $ \chi(x) = 2 n_{1}(x) $.
\end{lemma}



\begin{remark}
 	\label{re:remainder}    
	From the definition of $ \omega(x) $, it follows immediately that $ 0 \le \omega(x) < 2 \delta(x) $ for every $ x \in E_{1} $.
\end{remark} 

We now restrict our analysis to the circular guides of type A and B. 
    
\begin{lemma}
	\label{le:incr} Consider a circular guide of type A. There exists $ c=c(\alpha,r)>2 $ such that $ \chi(x) \le -c $ for every $ x \in E_{1} $ with $ \theta(x) \in (\bt,\pi-\bt) $ and $ n_{1}(x)>0 $.
\end{lemma}
\begin{proof}  
	By the symmetry of the guide, it is enough to prove the lemma for $ x \in E_{1} $ such that $ \theta(x) \in (\bt,\pi/2) $. For such values of $ x $, we have
	\begin{equation}
		\label{eq:der1}
		\delta'(\theta) = 1-\frac{\sin \theta}{\sqrt{r^{2}-\cos^{2} \theta}}<0
	\end{equation}
and
	\begin{equation}  
		\label{eq:der2}
		\delta''(\theta) = \frac{\cos \theta}{(r^{2}-\cos^{2} \theta)^{\frac{3}{2}}}(1-r^{2})>0. 
	\end{equation}
Since $ \delta'(\theta) \to -\infty $, as $ \theta \to \bt^{+} $, we can find $ \hth \in (\bt,\pi/2) $ such that $ \chi(x) < -3 $ for every $ x \in E_{1} $ with $ \theta(x) \in (\bt,\hth] $ and $ n_{1}(x)>0 $. 

We now consider the case $ x \in E_{1} $ with $ \theta(x) \in (\hth,\pi/2) $ and $ n_{1}(x)>0 $. It is trivial to see that for every $ \theta \in (\hth,\pi/2) $, we have $ \delta(\theta) = -\delta'(\theta) \Delta \theta $, where $ \Delta \theta $ is the length of the segment lying on the $ x $-axis whose endpoints are $ \theta $ and the intersection point of the tangent of $ \delta $ at $ (\theta,\delta(\theta)) $ with the $ x $-axis. Since $ \delta $ is strictly convex and $ \delta(\pi/2)=0 $, it follows that $ 0<\Delta \theta < \pi/2-\theta $ for every $ \theta \in (\hth,\pi/2) $. Hence
\begin{equation} 
	\label{eq:ratio}       
	\frac{\delta'(\theta)}{\delta(\theta)} = -\frac{1}{\Delta \theta} <-\frac{2}{\pi-2\theta} \qquad \text{for } \theta \in (\hth,\pi/2). 
\end{equation}

From \eqref{eq:ome},\eqref{eq:chi},\eqref{eq:ratio} and Remark \ref{re:remainder}, it follows that
\begin{align}
	\chi(x) & = \left(\alpha-\omega(x)\right) \frac{\delta'(\theta)}{\delta(\theta)} \notag \\ 
	& < -2\frac{\alpha-\omega(x)}{\pi-2\theta} \notag \\ 
	& < -2\frac{\alpha-2\delta(\theta)}{\pi-2\theta}
	\qquad \text{for } \theta \in (\hth,\pi/2). \label{eq:chiest}
\end{align}
	Let $ h(\alpha,\theta)=-2(\alpha-\delta(\theta))/(\pi-2\theta) $. Since $ \alpha \ge \pi $ and $ \delta(\theta)<\theta $, it is easy to see that $ \D_{\alpha}h<0 $ and $ h(\alpha,\hth) < -2 $. So $ h $ is strictly decreasing, and therefore
	\[
	\chi(x) < h(\alpha,\hth) < -2 \qquad \text{for } \theta \in (\hth,\pi/2).
	\]
	To complete the proof, set $ c=\min\{3,-h(\alpha,\hth)\} $.
\end{proof}
         
Since $ \delta' $ is strictly increasing for $ \theta \in (\bt,\pi/2) $ (see \eqref{eq:der2}), we have $ \delta'(x) < \delta'(\pi/2) = 1-1/r $ for every $ x \in M_{1} $ such that $ \theta(x) \in (\bt,\pi-\bt) $. This simple fact proves immediately the following lemma, saying that a result similar to Lemma \ref{le:incr} holds true for circular guides of type B.

\begin{lemma}
	\label{le:smallr}
	If $ r<1/2 $, then $ \chi(x) \le 2n_{1}(x)(1-1/r) < -2 $ for every $ x \in E_{1} $ such that $ \theta(x) \in (\bt,\pi-\bt) $ and $ n_{1}(x)>0 $.                                    
\end{lemma} 

\begin{remark}
It is precisely the fact that $ |\chi(x)|>2 $ proved in the previous lemmas that allows us to think of circular guides as optical devices having the property of focusing in a controlled way infinitesimal families of parallel rays entering the guide. In this sense, we can think of circular guides of type A and B as some sort of generalized absolutely focusing curves \cite{b2,do}.
This is the main ingredient of the construction of hyperbolic track billiards, and its proof will be completed in the next section.
\end{remark}

%
%
%
%
%
\section{Hyperbolicity} \label{se:hyperbolicity}

In this section, we prove that, under proper conditions concerning the circular guides and the distance between them, a track billiard admits an eventually strictly invariant cone field. By a well known result of Wojtkowski \cite{w2}, this property implies Theorem \ref{th:main}. 

\subsection{Focusing times} Recall that $ M $ is the phase space of the billiard in the track $ Q $. Given a tangent vector $ u \in T_{x} M $ at $ x \in \innt M $, let $ s \mapsto \gamma(s)=(q(s),v(s)) \in \innt M $ be a differentiable curve such that $ \gamma(0)=x $ and $ \gamma'(0)=u $. Next, define a family of rays $ s \mapsto \gamma_{+}(s) $ by setting $ \gamma_{+}(s)=\{q(s)+v(s):t \ge 0\} $. Similarly, define a second family of rays $ s \mapsto \gamma_{-}(s) $ by replacing $ \gamma $ with $ -\gamma $ in the definition of $ \gamma_{+} $. In geometrical terms, $ \gamma_{-} $ is obtained from $ \gamma_{+} $ by reflecting its rays at $ \D Q $. All the rays of $ \gamma_{+}(\gamma_{-}) $ intersect in linear approximation at one point along the ray $ \gamma_{+}(0)(\gamma_{-}(0)) $, which are called the \emph{focal points of $ u $}. If $ x=(s,\theta) $ and $ u=(ds,d\theta) $, then the distances between $ \pi(x) $ and the focal points of $ u $ lying on $ \gamma_{+}(0) $ and $ \gamma_{-}(0) $ are, respectively, given by 

\begin{equation}
	\label{eq:ftp} \fp(u) = \frac{\sin \theta}{\kappa(s)+ m(u)} 
\end{equation}
and 
\begin{equation}
	\label{eq:ftn} \fm(u) = \frac{\sin \theta}{\kappa(s)- m(u)}, 
\end{equation}
where $ \kappa(s) $ is the curvature $ \D Q $ at $ s $ and $ m(u)=d\theta/ds $ (see for example, \cite{w2}). We conventionally assume that the curvature of the outer circle is positive, whereas the curvature of the inner circle is negative. The distances $ \fp(u) $ and $ \fm(u) $ are called \emph{forward} and \emph{backward focusing times} of $ u $.
By summing the reciprocals of $ \fp(u) $ and $ \fm(u) $, we obtain the well known \emph{Mirror Formula}\footnote{The convention on the signs of the focusing times adopted here is different than that used in \cite{w2}.}
\begin{equation}
	\label{eq:mirro} \frac{1}{\fp(u)}+\frac{1}{\fm(u)}=\frac{2 \kappa(s)}{\sin \theta}. 
\end{equation}


\subsection{Fractional linear transformation} 

\begin{definition}
	Let $ E $ be the set of all collisions $ x \in M \setminus \snp{1} $ entering a circular guide of $ Q $. Also, for every $ x \in E $, denote by $ n(x) \ge 0 $ the times that the particle with initial state $ x $ hits the boundary of the circular guide before leaving it.
\end{definition}

Following \cite{w3}, we now introduce a transformation describing the relation between the focusing times of an infinitesimal family of billiard trajectories at the entrance and at the exit of a circular guide. 

Let $ x \in E $, and consider $ u \in \tax M $ with $ u \neq 0 $. Next, denote by $ \fx $ the map from the real projective line $ \R \cup \{\infty\} $ to itself given by $ \fin \mapsto \fo $, where $ \fin=\fm(u) $ and $ \fo=\fp(D_{x}T^{n(x)}u) $. Using the Mirror Formula, one can deduce that $ \fx $ is a linear fractional transformation (restricted to $ \R \cup \{\infty\} $)
\[
\fx(\fin)=\frac{a\fin + b}{c\fin + d},
\]
where $ a,b,c,d $ are real numbers such that
\[
ad-bc<0.
\]
This inequality is equivalent to the fact that the derivative $ \fx' $ is negative on $ \R $, and implies that the transformation $ \fx $ has two fixed points $ \tau_{1}(x) $ and $ \tau_{2}(x) $ on the real line. We will always assume that $ \tau_{1}(x) \ge \tau_{2}(x) $. 

The following lemma is an immediate consequence of the properties of $ \fx $.

\begin{lemma}
	\label{le:mfgeneral}
	Let $ x \in E $, and consider $ u \in \tax M $ with $ u \neq 0 $. Then
	\[
	 \fm(u)<\tau_{2}(x) \text{ or } \fm(u)>\tau_{1}(x) \Longleftrightarrow  
	\tau_{2}(x) < \fp(\dx T^{n(x)}u) < \tau_{1}(x).
	\]
\end{lemma}



\begin{definition}
Given a circular guide, define 
	\[
	\wt=\sup_{x \in E} \tau_{1}(x). 
	\] 
	The number $ \wt $ is called the \emph{focal length of the guide}. 
\end{definition} 


In the next theorem, we prove that the focal length of a circular guide of type A or B is always bounded above. 

\begin{theorem} 
	\label{th:boundtime} 
	We have
	\[ \wt \le \frac{\w{c}}{\w{c}-2},
	\]                               
	where $ \w{c} = c(\alpha,r) $ with $ c(\alpha,r) $ as in Lemma \ref{le:incr} for a guide of type A, and $ \w{c} = 2(1/r-1) $ for a guide of type B.
	\end{theorem}
\begin{proof}
We redefine $ M_{1} $ to be the set all collisions $ x \in M $ such that $ \pi(x) $ belongs to the outer circle of a circular guide of the track $ Q $, and $ E_{1} $ to be the set of all collisions $ x \in E \cap M_{1} $ such that the particle with initial state $ x $ hits the outer circle of the circular guide before leaving it. 

Let first assume that $ x = \sta \in E_{1} $, and compute the fixed points of the transformation $ \fx $, i.e., the solutions $ f \in \R \cup \{\infty\} $ of the equation 
\begin{equation}
\label{eq:fixed}
\fx(f)=f.
\end{equation}
If $ u \in \tax M $ with $ u \neq 0 $, then $ f=\fm(u)=\sin \theta/(1+m(u)) $ and $ \fx(f)=\fp(\dx T^{n_{1}(x)}u)=\sin \theta/(1+m(\dx T^{n_{1}(x)}u)) $. From \eqref{eq:der2}, it follows that $ m(\dx T^{n_{1}(x)}u)=m(u)/(1+\chi(x) m(u)) $, and so \eqref{eq:fixed} becomes 
\begin{equation*}
	\frac{m(u)}{1+\chi(x)m(u)} = - m(u),
\end{equation*}  
whose solutions are given by
\[
m(u)=0 \qquad \text{and} \qquad m(u)=-\frac{2}{\chi(x)}.
\]
From Lemmas \ref{le:disk},\ref{le:incr} and \ref{le:smallr}, we know that $ \chi(x)>0 $ if $ \theta \in (0,\bt) \cup (\pi-\bt,\pi) $, and $ \chi(x)<-2 $ if $ \theta \in (\bt,\pi-\bt) $ so that
\begin{equation*}
	\tau_{1}(x) = 
	\begin{cases}
		\sin \theta & \text{if } \theta \in (0,\bt) \cup (\pi-\bt,\pi), \\
		\frac{\sin \theta}{1+2/\chi(x)} & \text{if } \theta \in (\bt,\pi-\bt),
	\end{cases}
\end{equation*} 
and
\begin{equation*}
	\tau_{2}(x) = 
	\begin{cases}
	    \frac{\sin \theta}{1+2/\chi(x)} & \text{if } \theta \in (0,\bt) \cup (\pi-\bt,\pi), \\
	    \sin \theta & \text{if } \theta \in (\bt,\pi-\bt).
	\end{cases}
\end{equation*}
If $ \w{c} $ is defined as in the statement of the theorem, then by Lemmas \ref{le:disk},\ref{le:incr} and \ref{le:smallr}, we immediately obtain 
\begin{align*}
\sup_{x \in E_{1}} \tau_{1}(x) & \le \sup_{x \in E_{1}} \frac{\chi(x)}{\chi(x)+2} \\ & \le \frac{\w{c}}{\w{c}-2}. 
\end{align*}

Suppose now that $ x \in E \setminus E_{1} $. By increasing the central angle $ \alpha $ of the circular guide, we can always embed the orbit $ \{x,\ldots,T^{n(x)}x\} $ into an orbit $ \{y,\ldots,T^{n(y)}y\} $ of the enlarged guide\footnote{The argument presented here makes sense even when the enlarged guide has center angle greater or equal to $ 2 \pi $.} such that $ y \in E_{1} $ and $ n(x) \le n(y) \le n(x)+2 $. It follows that $ Ty=x \notin M_{1} $ or $ T^{n_{1}(y)-1}y=T^{n(x)}x $. Because of the symmetry of the problem, we can assume without a loss of generality that $ Ty=x $. We argue by contradiction, and suppose that 
\begin{equation}
	\label{eq:hypot} 
	\tau_{1}(x)>\tau_{1}(y). 
\end{equation}	
	Let $ u \in T_{y}M $ such that $ \fm(u)=\tau_{1}(x) $. Also, set $ w=\dx T^{-1}u $. Since $ \tau_{1}(x) $ is a fixed point of $ \fx $, we have 
	\begin{equation*}
		\tau_{1}(x) =
		\begin{cases}
			\fp(\dy T^{n_{1}(y)}w) & \text{if } T^{n(x)}x=T^{n_{1}y}y, \\
			\fp(\dy T^{n_{1}(y)-1}w) & \text{otherwise}.
		\end{cases}
	\end{equation*}
Using the Mirror Formula and the fact that $ \tau_{1}(y) $ is greater than the Euclidean distance $ d(\pi(x),\pi(y)) $ in $ \R^{2} $ between the points $ \pi(x) $ and $ \pi(y) $, we can easily show that $ 0<\fm(w)<\sin \theta = \tau_{2}(y) $, where $ y=\sta $. By Lemma \ref{le:mfgeneral}, it follows that
\begin{equation}
	\label{eq:uno}
	\tau_{2}(y)<\fp(\dx T^{n_{1}(y)}w)<\tau_{1}(y),
\end{equation}
and, using the Mirror Formula, 
\begin{equation}
	\label{eq:due}
-\sin \theta < \fp(\dx T^{n_{1}(y)-1}w) < d(\pi(x),\pi(y)) < \tau_{2}(y) < \tau_{1}(y).	
\end{equation}    
From \eqref{eq:uno} and \eqref{eq:due}, we then see that $ \tau_{1}(x)<\tau_{1}(y) $, contradicting our assumption \eqref{eq:hypot}. Hence, if we write $ \tau_{1}(x;\alpha) $ in place of $ \tau_{1}(x) $ to emphasize the dependence of $ \tau_{1} $ from the angle $ \alpha $, then, using the results obtained earlier for $ y \in E_{1} $, we get
\begin{align*} 
	\sup_{x \in E \setminus E_{1}} \tau_{1}(x;\alpha) & \le \sup_{\beta \ge \alpha} \sup_{y \in E_{1}} \tau_{1}(y;\beta) \\ 
	& \le \sup_{\beta \ge \alpha} \frac{\w{c}(\beta,r)}{\w{c}(\beta,r)-2}.
\end{align*} 
Next, note that the function $ \w{c} $ is increasing in $ \beta $ (because so is $ -h $; see the proof of Lemma \ref{le:incr}) for a guide of type A, and is independent of $ \beta $ for a guide of type B. Finally, observe that $ z/(z-2) $ is decreasing as a function of $ z<-2 $. Thus, we can conclude that  
\begin{equation*} 
	\sup_{x \in E \setminus E_{1}} \tau_{1}(x;\alpha) \le  \frac{\w{c}(\alpha,r)}{\w{c}(\alpha,r)-2},
\end{equation*}
which completes the proof.
\end{proof}

\subsection{Cone fields}  
\label{su:cf}   
A \emph{cone} in a 2-dimensional space $ V $ is a subset 
\[ \cf = \{aX_{1}+bX_{2}: ab \ge 0\}, \] where $ X_{1} $ and $ X_{2} $ are two linear independent vectors of $ V $. Equivalently, we can say that the cone $ \cf $ is a closed interval of the projective space $ \pr(V) $, the space of the lines in $ V $. The interior of $ \cf $ is defined by $ \innt \cf = \{aX_{1}+bX_{2}:ab>0\} $. Since the backward focusing time $ \fm $ and the forward focusing time $ \fp $ are both projective coordinates of $ \pr(\tax M)$, the set $ \cf = \{u \in \tax M: \fm(u)(\fp(u)) \in I \} $ is a cone in $ \tax M $ for every closed interval $ I \subset \R $. 
                            
Let $ \Lambda $ be a subset of $ \w{M} $ such that $ \mu(\Lambda)>0 $. Denote by $ \tl: \Lambda \to \Lambda $ the first return map on $ \Lambda $ induced by the billiard map $ T $. Also, denote by $ \mu_{\Lambda} $ the probability measure on $ \Lambda $ obtained by normalizing the restriction of $ \mu $ to $ \Lambda $. It is well known that the map $ \tl $ preserves $ \ml $.

\begin{definition}
A \emph{measurable cone field} $ \cf $ on $ \Lambda $ is a measurable map that associates to each $ x \in \Lambda $ a cone $ \cf(x) \subset \tax M $. We say that $ \cf $ is \emph{eventually strictly invariant} if for every $ x \in \Lambda $, we have 
\begin{enumerate}
	\item $ \dx \tl \cf(x) \subset \cf(\tl x) $,
	\item $ \exists $ an integer $ k(x)>0 $ such that $ \dx \tl^{k(x)} \cf(x) \subset \innt \cf(\tl^{k(x)}x) \cup \{0\} $.
\end{enumerate}
\end{definition}

\begin{remark}
	\label{re:induced}
	By \cite{w2}, the existence of such a cone field (plus other properties, always satisfied by track billiards) implies that $ \tl $ is hyperbolic. Furthermore, if the set $ \cup_{k \in \Z} T^{k} \Lambda $ has full $ \mu $-measure, then it is not difficult to see that $ T $ is hyperbolic as well (see \cite{w1}).
\end{remark} 

We now define an invariant cone field for circular track billiards. In the next subsection, we will show, relying on Lemmas \ref{le:incr}, \ref{le:smallr} and Theorem \ref{th:boundtime}, that this cone field is eventually strictly invariant if the straight guides of a track are sufficiently large. 

Let $ \w{E} = E \cap \w{M} $ be the set of entering collisions with infinite positive and negative semi-orbits. We define a measurable cone field on $ \w{E} $ as follows 
\begin{equation}
	\label{eq:cone} 
	\cf(x)=\{u \in T_{x} M: \fm(u) \ge \wt(x) \} \qquad \text{for all } x \in \w{E},
\end{equation}
where $ \wt(x) $ is the focal length of the circular guide containing $ \pi(x) $. The cone field  $ \cf $ is continuous (and therefore measurable), because so is $ \wt(x) $.


\subsection{Cone fields and integrability}
This subsection is intended to provide a more direct description of the construction of the cone field $ \cf $ in \eqref{eq:cone}, and to clarify the role  of the integrability of the billiard dynamics inside circular guides in this construction. Let $ T_{1} $ and $ E_{1} $ be, respectively, the tranformation and the set of entering collisions as in Subsection \ref{su:dyncurve}. We recall that $ E_{1} $ consists of all entering collisions $ x \in E $ such that $ x $ and the last collision of the orbit of $ x $ with the circular guide belong to $ M_{1} $. We will restrict the following analysis to the set $ E_{1} $, since the basic idea behind the construction of $ \cf $ remains the same on $ E \setminus E_{1} $.


We start by observing that there exists a natural cone field $ C $ that is invariant along the orbits of $ T_{1} $. This is given by
\begin{equation}
	\label{eq:cancone}
	C(x) = 
	\begin{cases}
		\left\{a \frac{\D}{\D s} + b \frac{\D}{\D \theta}: ab \ge 0 \right\} & \text{if } x \in \hat{M}_{1}, \\
	    \left\{a \frac{\D}{\D s} + b \frac{\D}{\D \theta}: ab \le 0 \right\} & \text{if } x \in M_{1} \setminus \hat{M}_{1},
	\end{cases}   	
\end{equation}  
where $ \hat{M}_{1} $ is the set of all collisions $ x \in M_{1} $ such that $ \theta(x) \in (0,\bt) \cup (\pi-\bt,\pi) $. The invariance of $ C $ is a consequence of the invariance of $ \D_{s} $ (which in turn is a consequence of invariance of the angular momentum of the particle inside a circular guide) and the twist $ 2 \delta' $ of $ T_{1} $, which is responsible for tilting the `vertical' vector $ \D_{\theta} $ to the right or to the left according to the twist's sign. Note that $ \delta' > 0 $ on $ \hat{M}_{1} $ and $ \delta'<0 $ on $ M_{1} \setminus \hat{M}_{1} $. 

The cone field $ \cf $ is obtained by modifying properly the cones of $ C $. Here properly means that after such a modification, the new cone field $ \cf $ must have the property, which we will call (*), that there are two real numbers $ \tau_{-} $ and $ \tau_{+} $ such that for every $ x \in E_{1} $, 
\begin{enumerate}[(i)]
	\item $ [\tau_{-},+\infty] \subset \fm(\cf(x)) $,
	\item $ \fp\left(\dx T^{n_{1}(x)}_{1} \cf(x)\right) \subset [-\infty,\tau_{+}] $. 
\end{enumerate} 
These conditions mean that each cone $ \cf(x) $ must consist of tangent vectors (corresponding to infinitesimal families of billiard orbits) such that their backward focusing time varies between $ \tau_{-} $ and $ +\infty $ at the entrance of the guide, and their forward focusing time varies between $ -\infty $ and $ \tau_{+} $ at the exit of the guide. We point out that focusing curves having an invariant cone field with this property (for dispersing curves, such a cone field always exists) play a crucial role in designing hyperbolic billiards \cite{b2,do,m2,w2,w3}. Indeed, once we have selected some of these special curves, to obtain a hyperbolic billiard domain, all that we need to do is to arrange them, maybe using some straight lines, so that there is sufficient distance between any pair of them. This recipe remains valid if boundary components are replaced by circular guides with an invariant cone field that has Property (*). 

It is easy to check that $ C $ enjoys Property (*) on $ \hat{M}_{1} $. In fact, using Formulae \eqref{eq:ftp} and \eqref{eq:ftn}, we obtain 
\[ \fm(C(x)) = [-\infty,0] \cup [\sin \theta(x),+\infty] \] and 
\[ \fp\left(\dx T^{n_{1}(x)}_{1} C(x)\right) \subset \fp(C(x)) = [0,\sin \theta(x)] \] for every $ x \in \hat{M}_{1} $ (as in Subsection \ref{su:dyncurve}, we are assuming that the radius of the outer circle is equal to one).
Property (*) is however not satisfies by $ C $ on $ M_{1} \setminus \hat{M}_{1} $. More precisely, while part (ii) of (*) holds true, because Lemma \ref{le:incr} (for circular guides of type A) and Lemma \ref{le:smallr} (for circular guides of type B) imply that for every $ x \in M_{1} \setminus \hat{M}_{1} $,                                   
\[ -1/2 < m(\dx T^{n_{1}(x)}_{1} \D_{\theta}) < 0, \] and so \[ \fp\left(\dx T^{n_{1}(x)}_{1} C(x)\right) \subset [\sin \theta(x),2 \sin \theta(x)], \] part (i) is not satisfied, because \[ +\infty \notin \fm(C(x)) = [0,\sin \theta(x)] \qquad \text{for all } x \in M_{1} \setminus \hat{M}_{1}. \]
This problem can be easily solved by replacing the vertical edge $ \D_{\theta} $ of $ C(x) $ with a vector $ X(x) \in \tax M_{1} $ such that $ 0<m(X(x))<1 $. The vector $ X(x) $ has to be chosen so that part (ii) of (*) remains valid, being the new cones wider than the old ones. More precisely, we have to show that there exist a real number $ c>0 $ and a vector $ X(x) \in \tax M_{1} $ such that for every $ x \in E_{1} \setminus \hat{M}_{1} $,
\[ 1-c < m(X(x)) < 1 \qquad \text{and} \qquad m\left(\dx T^{n_{1}(x)}_{1} X(x)\right)>-1+c. \] It is not difficult to see that this property implies both parts (i) and (ii) of (*) with some $ \tau_{-} $ and $ \tau_{+} $ less than $ 1/c $.

The existence of such $ c $ and $ X $ for $ x \in E_{1} $ is proved in Lemmas \ref{le:incr} and \ref{le:smallr}. In Theorem \ref{th:boundtime}, we extends this result to all points of $ E $, and also provide a specific choice for the vector $ X $, which is determined (up to a positive scalar factor) by the relation $ \fm(X(x))=\wt(x) $ for $ x \in E $. Here $ \wt(x) $ is the focal length of the circular guide containing $ \pi(x) $. 

Finally, we remark that the second edge of the cone field in \eqref{eq:cone} is not $ \D_{s} $ as in \eqref{eq:cancone}, but $ \D_{s} + \D_{\theta} $ if $ x \in M_{1} $, and $ -\D_{s}+r^{-1} \D_{\theta} $ otherwise. This makes that cone field in \eqref{eq:cone} narrower than the $ \cf $ constructed in this subsection, but it is easy to check that Property (*) remains valid for it.

\subsection{Hyperbolicity}
Let $ Q $ be a track, and assume that its guides are ordered in such a way that the $ i $th straight guide connects the $ i $th and $ (i+1) $th circular guides. The $ (n+1) $th circular guide coincides with the first one so that there are exactly $ n $ circular guides separated by $ n $ straight guides. We also assume that each circular guide is either of type A or B. For every $ 1 \le i \le n $, let $ \w{\tau}_{i} $ and $ l_{i} $ be the focal length and the length of the $ i $th circular guides and the $ i $th straight guide, respectively. We say that such a track $ Q $ satisfies Condition H if the distance between any pair of consecutive circular guides of $ Q $ is greater than the focal length of the two circular guides, i.e.,
\begin{equation*} \label{eq:h}
	l_{i}>\wt_{i}+\wt_{i+1} \qquad \text{for each } i=1,\ldots,n. \tag{H}
\end{equation*}

We can now give the precise formulation and the proof of Theorem \ref{th:main}, the main result of this paper.

\begin{theorem} \label{th:hyper}
	Suppose that a track $ Q $ satisfies Condition H. Then the billiard map $ T $ in $ Q $ is hyperbolic.
\end{theorem}
\begin{proof} 
	By Remark \ref{re:induced}, it is enough to prove that the cone field $ \cf $ defined in \eqref{eq:cone} is eventually strictly invariant, and the set $ \cup_{k \in \Z} T^{k} \w{E} $ has full $ \mu $-measure.
	
	Let $ x \in \w{E} $, and consider $ u \in \cf(x) $ with $ u \neq 0 $. By definition of $ \cf(x) $, we have $ \fm(u) > \wt(x) \ge \tau_{1}(x) $ so that Lemma \ref{le:mfgeneral} implies that $ \fp(\dx T^{n(x)}u)<\tau_{1}(x) \le \wt(x) $. Now, note that $ T^{n(x)}x $ is a collision leaving a circular guide, and that the piece of the orbit of $ x $ between $ x $ and $ T_{\w{E}}x $ crosses a straight guide of length $ l $. By Condition H, we then have $ l>\wt(x)+\wt(T_{\w{E}}x) $, and hence
\begin{align*}
	\fm(\dx T_{\w{E}} u) & = l - \fp(\dx T^{n(x)}u) \\
	& \ge l - \wt(x) \\
	& > \wt(T_{\w{E}}x).
\end{align*}
This means that $ \dx T_{\w{E}} u \in \innt \cf(T_{\w{E}}x) $, and we can conclude that $ \cf $ is eventually strictly invariant with $ k(x)=1 $ for every $ x \in \w{E} $. It is clear that $ \cup_{k \in \Z} T^{k} \w{E} = \w{M} \setminus N $ (for the definition of $ N $, see Subsection \ref{su:inv}). Since $ \mu(N)=0 $, it follows that $ \cup_{k \in \Z} T^{k} \w{E} $ has full measure.
\end{proof}    

\begin{remark}
	It is easy to check that the so called Monza billiard considered in \cite{vpr} satisfies Condition H. Note that its circular guides are of type B. Theorem \ref{th:hyper} then assures that the Monza billiard is hyperbolic.
\end{remark}

\section{3-dimensional track billiards} 
\label{se:td}  

In this section, we introduce 3-dimensional track billiards, and extend Theorem \ref{th:hyper} to them.                                                      


\begin{definition} 
	A \emph{3-dimensional cylindrical (straight) guide} $ \w{G} $ is the direct product $ G \times I $, where $ G $ is a 2-dimensional circular (straight) guide $ G $, and $ I \subset \R $ is a closed interval. Furthermore, we assume that $ G $ is of type A or B.
\end{definition}
  
\begin{definition}
We say that a domain $ \w{Q} \subset \R^{3} $ is a \emph{3-dimensional track} if there exist a differential Jordan curve $ \gamma $ in $ \R^{3} $ and a rectangle $ R $ such that the intersection of $ \w{Q} $ with the plane orthogonal to the tangent line $ \tax \gamma $ is equal to $ R $ for every $ x \in \gamma $. We further require $ \w{Q} $ to be an union of finitely many alternating cylindrical and straight guides.	
\end{definition}

An example of a 3-dimensional track is depicted in Fig. \ref{fig:bent}.
\begin{figure}[t] 
	\begin{center}
		\includegraphics[width=6cm]{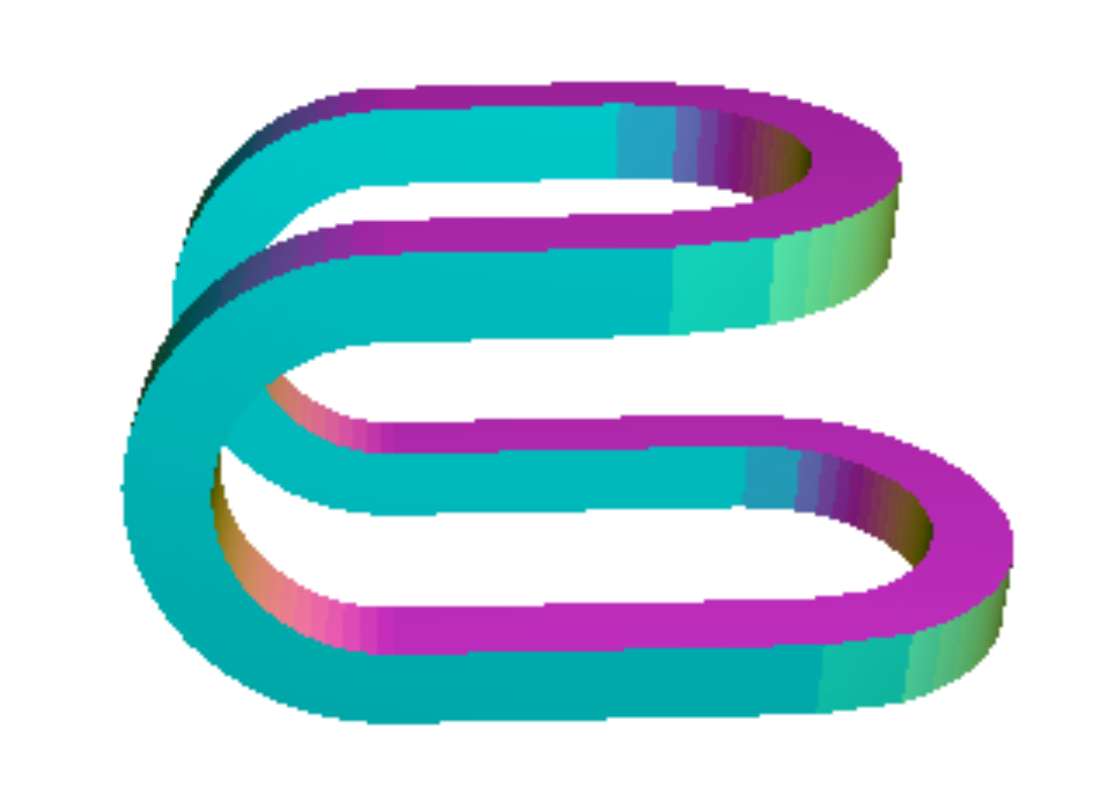}
	\end{center}
	\caption{A 3-dimensional track that satisfies Condition \~H.} 
	\label{fig:bent}
\end{figure}

\begin{remark}
	If we denote by $ v_{*} $ the orthogonal projection of the velocity of the particle along the oriented tangent of $ \gamma $, then, as for 2-dimensional track billiards, $ sgn(v_{*}) $ is constant along the trajectory of the particle. In this way, we see that the billiard phase space of a 3-dimensional track is partitioned into the three invariant sets consisting of collision states such that $ sgn(v_{*})>0 $, $ sgn(v_{*})<0 $ and $ sgn(v_{*})=0 $, respectively. 
\end{remark}  

\begin{definition}
	Let $ \w{P} $ be a subdomain of a 3-dimensional track, which consists of two cylindrical guides $ \w{G}_{1} $ and $ \w{G}_{2} $ connected by a straight guide such that circular guides $ G_{1} $ and $ G_{2} $ lie on orthogonal planes (i.e., their normals are orthogonal) of $ \R^{3} $ (Fig. \ref{fig:twist}).
\end{definition}
\begin{figure}[t] 
	\begin{center}
  	\includegraphics[width=6cm]{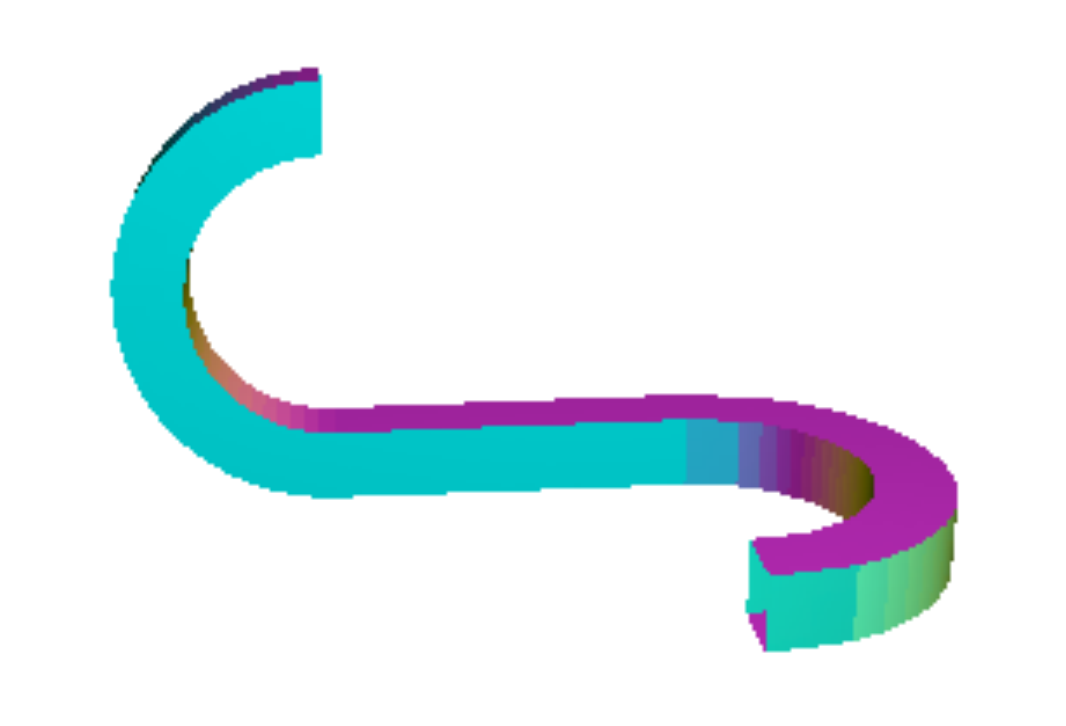}
	\end{center}
	\caption{A twisted guide} 
	\label{fig:twist}
\end{figure}

If a track does not contain a twisted guide, then all the 2-dimensional tracks are contained in a single plane. It follows that the momentum of the particle along the normal of that plane is a first integral of motion, and so the billiard is not completely hyperbolic. In the 2-dimensional case, we managed to prove that track billiards are hyperbolic if the satisfy Condition H. The 3-dimensional analogue of Condition H reads as follows. Let $ \w{Q} $ be a track such that $ \w{Q} $ is an union of finitely many cylindrical and straight guides. We say that $ \w{Q} $ satisfies Condition \~H if 
\begin{enumerate} 
	\item the distance between any two cylindrical guides $ \w{G}_{1} $ and $ \w{G}_{2} $ is greater than $ \tau_{1}+\tau_{2} $, where $ \tau_{1} $ and $ \tau_{2} $ are the focal lengths of the 2-dimensional guides corresponding to $ \w{G}_{1} $ and $ \w{G}_{2} $;
	\item $ \w{Q} $ contains at least one twisted guide.  
\end{enumerate}  

An example of track satisfying \~H is shown in Fig. \ref{fig:bent}. Billiards in tracks satisfying Condition \~H are closely related to certain hyperbolic semi-focusing cylindrical billiards \cite{bd1,bd2}, and are examples of twisted Cartesian products \cite{w3}. Theorem \ref{th:hyper} combined with the results of \cite{bd1} (or Theorem 17 of \cite{w3}) implies that for a 3-dimensional track billiard satisfying Condition \~H, there exists an invariant cone field that is strictly invariant along every orbit crossing a twisted guide, thus proving the following theorem.                                

\begin{theorem}
	 If a 3-dimensional track satisfies Condition \~H, then billiard map in such a track is hyperbolic.
\end{theorem}


\begin{thebibliography}{AAAAAAA} 

\bibitem[B1]{b1} L. Bunimovich, \emph{A Theorem on Ergodicity of Two-Dimensional Hyperbolic Billiards}, Comm. Math. Phys. \textbf{130} (1990), 599-621.

\bibitem[B2]{b2} L.~Bunimovich, \emph{On absolutely focusing mirrors}, Ergodic theory and related topics, III (G\"ustrow, 1990), Lect. Notes Math. \textbf{1514}, Springer-Verlag 1992, 62-82.

\bibitem[B3]{b3} L. Bunimovich, \emph{Mushrooms and other billiards with divided phase space}, Chaos \textbf{11} (2001), 802-808.

\bibitem[B-D1]{bd1} L.~Bunimovich, G.~Del~Magno, \emph{Semi-focusing billiards: hyperbolicity}, Comm. Math. Phys. \textbf{262} (2006), 17-32.

\bibitem[B-D2]{bd2} L.~Bunimovich, G.~Del~Magno, \emph{Semi-focusing billiards: ergodicity}, Ergodic Theory Dynam. Systems, to appear.

\bibitem[B-L]{bl} L. Bussolari, M. Lenci, \emph{Hyperbolic billiards with nearly flat focusing boundaries}, Physica D, to appear.

\bibitem[C-M]{cm} N. Chernov, R. Markarian, \emph{Chaotic billiards}, Mathematical Surveys and Monographs \textbf{127}, AMS, 2006.

\bibitem[C-F-S]{cfs} I.~Cornfeld, S.~Fomin, Ya.~Sinai, \emph{Ergodic theory}, Springer-Verlag, New York, 1982.

\bibitem[C-D-F-K]{cdfk} B. Chenaud, P. Duclos, P. Freitas, D. Krej\v ci\v r'k, \emph{Geometrically induced discrete spectrum in circular tubes}, Differential Geom. Appl. \textbf{23} (2005), 95-105.

\bibitem[D-M]{dm} G. Del Magno, R. Markarian, preprint.

\bibitem[D]{do} V.~Donnay, \emph{Using integrability to produce chaos: billiards with positive entropy}, Comm. Math. Phys. \textbf{141} (1991), 225-257.

\bibitem[E-S]{es} P. Exner, P. \v Seba, \emph{Bound states in curved quantum waveguides}, J. Math. Phys. \textbf{30} (1989), 2574-2580.

\bibitem[G-J]{gj} J. Goldstone, R.L. Jaffe, \emph{Bound states in twisting tubes}, Phys. Rev. B \textbf{45} (1992), 14100-14107.

\bibitem[H-P]{hp} M. Horvat, T. Prosen, \emph{Uni-directional transport properties of a serpent billiard}, J. Phys. A: Math. Gen. \textbf{37} (2004), 3133-3145.

\bibitem[K-S]{ks} A. Katok, J.-M. Strelcyn, \emph{Invariant manifolds, entropy and billiards; smooth maps with singularities}, Lect. Notes Math. {\bf 1222}, Springer, New York, 1986.

\bibitem[M1]{m2} R. Markarian, \emph{Non-uniformly hyperbolic billiards}, Ann. Fac. Sci. Toulouse Math. \textbf{6} 3 (1994), 223-257. 

\bibitem[P]{p} R. Peirone, \emph{Billiards in Tubular Neighborhoods of Manifolds of Codimension 1}, Comm. Math. Phys. \textbf{207} (1999), 67-80.

\bibitem[V-P-R]{vpr} G. Veble, T. Prosen, M. Robnik, \emph{Expanded boundary integral method and chaotic time-reversal doublets in quantum billiards}, New J. Phys. \textbf{9} (2007), 15.

\bibitem[W1]{w1} M.~Wojtkowski, \emph{Invariant families of cones and Lyapunov exponents}, Ergodic Theory Dynam. Systems \textbf{5} (1985), 145-161. 

\bibitem[W2]{w2} M.~Wojtkowski, \emph{Principles for the design of billiards with nonvanishing Lyapunov exponents}, Comm. Math. Phys. \textbf{105} (1986), 391-414.

\bibitem[W3]{w3} M.~Wojtkowski, \emph{Design of hyperbolic billiards}, Comm. Math. Phys. \text{273} (2007), 283-304.
         
\end{thebibliography}
\end{document}